\documentclass{amsart}

\newcommand{\R}{\mathbb{R}}

\newcommand{\N}{\mathbb{N}}

\newcommand{\LP}{\mathcal{LP}}

\theoremstyle{plain}
\newtheorem{theorem}{Theorem}[section]

\newtheorem{corollary}[theorem]{Corollary}

\theoremstyle{definition}

\newtheorem{example}{Example}

\theoremstyle{remark}

\numberwithin{equation}{section}

\begin{document}

\title{Extended Laguerre inequalities and a criterion for real zeros}

\author{David~A.~Cardon}

\address{Department of Mathematics\\
Brigham Young University\\
Provo, Utah 84602, USA\\
E-mail: cardon@math.byu.edu\\
http://www.math.byu.edu/cardon}

\begin{abstract}
Let $f(z)=e^{-bz^2}f_1(z)$ where $b \geq 0$ and $f_1(z)$ is a
real entire function of genus $0$ or $1$. We give a necessary
and sufficient condition in terms of a sequence of inequalities
for all of the zeros of $f(z)$ to be real. These inequalities
are an extension of the classical Laguerre inequalities.
\end{abstract}


\maketitle

\section{Introduction}

The Laguerre-P\'olya class, denoted $\LP$, is the collection of
real entire functions obtained as uniform limits on compact
sets of polynomials with real coefficients having only real
zeros. It is known that a function $f$ is in $\LP$ if and only
if it can be represented as
\begin{equation} \label{eqn:WeierstrassProduct}
f(z) = e^{-b z^2} f_1(z)
\end{equation}
where $b \geq 0$ and where $f_1(z)$ is a real entire function
of genus $0$ or $1$ having only real zeros. The basic theory of
$\LP$ can be found in \cite[Ch.~8]{Levin1980} and
\cite[Ch.~5.4]{RahmanSchmeisser2002}.

In this paper, we extend a theorem of Csordas, Patrick, and
Varga on a necessary and sufficient condition for certain real
entire functions to belong to the Laguerre-P\'olya class. They
proved the following:

\begin{theorem}
 \label{thm:CsordasVargaPatrick}
Let
\[
f(z)=e^{-bz^2}f_1(z), \qquad (b \geq 0, \,f(z) \not\equiv 0),
\]
where $f_1(z)$ is a real entire function of genus $0$ or $1$.
Set
\begin{equation} \label{eqn:Ln}
L_n[f](x) = \sum_{k=0}^{2n} \frac{(-1)^{k+n}}{(2n)!} \binom{2n}{k} f^{(k)}(x) f^{(2n-k)}(x)
\end{equation}
for $x \in \R$ and $n \geq 0$. Then $f(z) \in \LP$ if and only if
\begin{equation}  \label{eqn:Lnpositive}
L_n[f](x) \geq 0
\end{equation}
for all $x \in \R$ and all $n \geq 0$.
\end{theorem}
The forward direction is due to Patrick~\cite[Thm.
1]{Patrick1973}. The reverse direction was proved by Csordas
and Varga~\cite[Thm.~2.9]{CsordasVarga1990}.
Theorem~\ref{thm:CsordasVargaPatrick} is significant because it
gives a nontrivial sequence of inequality conditions that hold
for functions in the Laguerre-P\'olya class. The case $n=1$
reduces to the classical Laguerre inequality which says that if
$f(z) \in \LP$, then
\[
[f'(x)]^2 - f(x) f''(x) \geq 0
\]
for $x \in \R$. Consequently, inequalities like those in
Theorem~\ref{thm:CsordasVargaPatrick} are sometimes called
Laguerre-type inequalities. Csordas and Escassut discuss the
inequalities $L_n[f](x) \geq 0$ and related Laguerre-type
inequalities in~\cite{CsordasEscassut2005}. Other results
on similar inequalities of Tur\'an and Laguerre types can be
found in~\cite{CardonRich2008} and~\cite{CravenCsordas2002}.

\section{An extension of Laguerre-type inequalities}

In this section, we extend
Theorem~\ref{thm:CsordasVargaPatrick} and give new necessary
and sufficient inequality conditions for a function to belong
to the Laguerre-P\'olya class.

First we generalize the operator $L_n$ defined in
Theorem~\ref{thm:CsordasVargaPatrick}. Let
\begin{equation}\label{eqn:g}
g(z) = \sum_{\ell=0}^M c_{\ell} z^{\ell} = \prod_{j=1}^M (z+\alpha_j)
\end{equation}
be a polynomial with complex roots. Define
$\Phi(z,t)$ as the product
\begin{equation}
\Phi(z,t) = \prod_{j=1}^M f(z+\alpha_j t).
\end{equation}
The coefficients of the Maclaurin series of $\Phi(z,t)$ with
respect to $t$ are functions of $z$, and we write:
\begin{equation}
\Phi(z,t) = \sum_{k=0}^{\infty} A_k(z) t^k,
\end{equation}
where
\begin{equation} \label{eqn:Ak}
A_k(z) = \frac{1}{k!}\Big[ \frac{d^k}{dt^k} \Phi(z,t)\Big]_{t=0}
=\frac{1}{k!}\Big[\frac{d^k}{dt^k} \prod_{j=1}^M f(z+\alpha_j t)\Big]_{t=0}.
\end{equation}
Since $f(z)$ is entire, each $A_k(z)$ is entire. Another
expression for $A_k(z)$ is given in~\eqref{eqn:Akexpanded}. The
choice $g(z)=1+z^2=(z-i)(z+i)$ produces $A_{2k+1}(z)=0$ and
$A_{2k}(z)=L_k[f](z)$ as in \eqref{eqn:Ln} of
Theorem~\ref{thm:CsordasVargaPatrick}. Thus, we may regard the
sequence of functions $A_k(z)$ as a generalization of the
sequence $L_k[f](z)$. We note that the zeros of $A_k(z)$ were
studied by Dilcher and Stolarsky
in~\cite{DilcherStolarsky1992}. In \S\ref{section:Examples}, we
give several examples of $A_k(z)$ for interesting choices of
$g(z)$.

\begin{theorem} \label{theorem:main}
Let $f(z) = e^{-b z^2} f_1(z)$, where $f_1(z) \not\equiv0$ is a
real entire function of genus $0$ or $1$ and $b \geq 0$. Assume
$g(z)$ in \eqref{eqn:g} is an even polynomial
with non-negative real coefficients having at least one
non-real root. Then $f \in \LP$ if and only if
\begin{equation}
A_k(x) \geq 0
\end{equation}
for all $x \in \R$ and all $k\geq 0$.
\end{theorem}
\begin{corollary}
The choice $g(z)=1+z^2$ in Theorem~\ref{theorem:main} gives
$f(z) \in \LP$ if and only if $L_k[f](x) \geq 0$ for all $x \in
\R$ and all $k \geq 0$, as stated in
Theorem~\ref{thm:CsordasVargaPatrick}.
\end{corollary}

\begin{proof}[Proof of Theorem~\ref{theorem:main}]
Since $g(z)$ is an even polynomial, it follows that $\alpha_j$
is a root if and only if $-\alpha_j$ is a root with the same
multiplicity. So,
\[
\Phi(z,t) = \prod_{j=1}^M f(z+\alpha_jt) = \prod_{j=1}^M f(z-\alpha_j t) = \Phi(z,-t).
\]
Hence, $A_k(z)\equiv 0$ for all odd $k$ and we may write
\[
\Phi(z,t) = \sum_{k=0}^{\infty} A_{2k}(z) t^{2k}.
\]

Now assume $A_{2k}(x) \geq 0$ for all $x \in \R$ and all $k
\geq 0$. Let $f(z)=e^{-b z^2}f_1(z)$ where $b\geq 0$ and
$f_1(z)$ is a real entire function of genus $0$ or $1$, and
assume $f(z)$ is not identically zero. Suppose, by way of
contradiction, that $f(z)$ has a non-real root, say $z_0$. Let
$\alpha_s$ be any fixed non-real root of $g(z)$ and write
\[
z_0 = x_0 + \alpha_s t_0,
\]
where both $x_0$ and $t_0$ are real. Then
$f(z_0)=f(x_0+\alpha_s t_0)=0$, and
\[
0 = \Phi(x_0,t_0) = \prod_{j=1}^M f(x_0+\alpha_j t_0)
= \sum_{k=0}^{\infty} A_{2k}(x_0) t_0^{2k}.
\]
Assume $t \neq 0$. Then the nonnegativity of $A_{2k}(x_0)$
implies $A_{2k}(x_0)=0$ for all $k$. This in turn implies
$\Phi(x_0,t)$ is identically zero for all complex $t$. But that
is false since $f(z)$ is a nonzero entire function. Therefore,
$t_0=0$. Then $z_0 = x_0+\alpha_s t_0 = x_0$ is also real,
contradicting the choice of $z_0$. Thus, all the roots of
$f(z)$ are real and $f(z) \in \LP$.

Conversely, assuming $f(z) \in \LP$, we will show that
$A_{2k}(x) \geq 0$ for all $x \in \R$ and all $k \geq 0$. We
will show this when $f(z)$ is a polynomial and the result for
arbitrary $f(z) \in \LP$ will follow by taking limits. Let
\[
f(z) = \prod_{i=1}^n(z+r_i)
\]
where $r_1,\ldots,r_n$ are real. Calculating $\Phi(z,t)$ gives
\begin{align}
\Phi(z,t)
& = \prod_{j=1}^M f(z+\alpha_j t)  = \prod_{j=1}^M \prod_{i=1}^n(z+\alpha_j t + r_i) \nonumber \\
& = \prod_{i=1}^n \prod_{j=1}^M \big( (z+r_i) +\alpha_j t\big)
= \prod_{i=1}^n \sum_{\ell=0}^M c_{\ell}(z+r_i)^{\ell} t^{M-\ell},   \label{eqn:Phiproductsum}
\end{align}
where $g(z) = \prod_{j=1}^M(z+\alpha_j) = \sum_{\ell=0}^M
c_{\ell} z^{\ell}$. Since $g(z)$ is an even polynomial,
$c_{\ell}=0$ for odd $\ell$ and
\begin{equation} \label{eqn:polynomialcasePhi}
\Phi(z,t)=\prod_{i=1}^n \sum_{\ell=0}^{M/2} c_{2\ell}(z+r_i)^{2\ell}t^{M-2\ell}
= \sum_{k=0}^{nM/2} A_{2k}(z) t^{2k}.
\end{equation}
From \eqref{eqn:polynomialcasePhi}, $A_{2k}(z)$ is the sum of
products of terms of the form $c_{2\ell}(z+r_i)^{2\ell}$.
Because $c_{2\ell} \geq 0$ and $(z+r_i)^{2\ell}$ is a square,
it follows that $A_{2k}(x) \geq 0$ for real $x$.

Now let $f(z) \in \LP$ be an arbitrary function that is not a
polynomial. Then there exist polynomials $f_n(z) \in \LP$ such
that
\[
\lim_{n \rightarrow \infty} f_n(z) = f(z)
\]
uniformly on compact sets. The derivatives also satisfy
\[
\lim_{n \rightarrow \infty} f_n^{(k)}(z) = f^{(k)}(z)
\]
uniformly on compact sets. If we write
\[
\Phi_n(z,t) = \prod_{j=1}^M f_n(z+\alpha_j t) = \sum_{k=0}^{\infty} A_{n,2k}(z) t^{2k},
\]
we see from \eqref{eqn:Ak} that
\[
\lim_{n \rightarrow \infty} A_{n,2k}(z) = A_{2k}(z)
\]
uniformly on compact sets. Since $A_{n,2k}(x) \geq 0$ for real
$x$, the limit also satisfies this inequality. Thus,
for arbitrary $f(z) \in \LP$, $A_{2k}(x) \geq 0$ for $x \in\R$
and $k \geq 0$, completing the proof of the theorem.
\end{proof}

\section{Discussion and Examples}\label{section:Examples}

The function $A_k(z)$ in \eqref{eqn:Ak} is described in terms of
the $k$th derivative of a product of entire functions. Either
by using the generalized product rule for derivatives or by
expanding each $f(z+\alpha_j t)$ as a series and multiplying
series, one obtains the following formula for $A_k(z)$:
\begin{equation} \label{eqn:Akexpanded}
A_k(z)
=\sum_{\lambda \vdash k} \frac{m_{\lambda}(\alpha_1,\ldots,\alpha_M)}{\lambda_1!\cdots\lambda_r!}
f(z)^{M-r} \prod_{j=1}^{r} f^{(\lambda_j)}(z),
\end{equation}
where $\lambda \vdash k$ means that the sum is over all
unordered partitions $\lambda$ of $k$,
\[
\lambda=(\lambda_1,\ldots,\lambda_r) \qquad k = \lambda_1+\cdots+\lambda_r,
\]
where $r$ is the length of the partition $\lambda$, and
where $m_{\lambda}(\alpha_1,\ldots,\alpha_M)$ is the monomial
symmetric function of $M$ variables for the partition $\lambda$
evaluated at the roots $\alpha_1,\ldots,\alpha_M$.

The coefficients $c_{\ell}$ of $g(z)=\sum_{\ell=0}^M
c_{\ell}z^{\ell}=\prod_{j=1}^M(z+\alpha_j)$ are elementary
symmetric function of $\alpha_1,\ldots,\alpha_M$. The monomial
symmetric functions $m_{\lambda}(\alpha_1,\ldots,\alpha_M)$
appearing in \eqref{eqn:Akexpanded} can therefore be calculated
in terms of $c_0,\ldots,c_M$ without direct reference to
$\alpha_1,\ldots,\alpha_M$. We see that if $c_0,\ldots,c_M$ are
real, then $m_{\lambda}(\alpha_1,\ldots,\alpha_M)$ is also
real. However, in general
$m_{\lambda}(\alpha_1,\ldots,\alpha_M)$ is not necessarily
positive even if all the $c_{\ell}$ are positive. So, in the
setting of Theorem~\ref{theorem:main}, the type of summation
appearing in \eqref{eqn:Akexpanded} will typically involve both
addition and subtraction, and the nonnegativity of $A_k(x)$ for
real $x$ is not directly obvious from this representation.

\begin{example}
Let $g(z)=4+z^4$ and let $f(z) \in \LP$. Then $\Phi(z,t)$ is
\[
f(z+(1+i)t)f(z+(1-i)t)f(z+(-1+i)t)f(z+(-1-i)t) = \sum_{k=0}^{\infty} A_{2k}(z)t^{2k}.
\]
A small calculation shows that
\begin{multline*}
\tfrac{3}{2}A_4(x)=-f(z)^3 f^{(4)}(z)+3 f(z)^2 f''(z)^2+6 f'(z)^4 \\
+4 f(z)^2 f^{(3)}(z) f'(z) -12 f(z) f'(z)^2 f''(z).
\end{multline*}
According to Theorem~\ref{theorem:main} this expression is
nonnegative for all $x \in \R$.
\end{example}

\begin{example}
Let $f(z)=\prod_{i=1}^n(z+r_i)$ where $r_1,\ldots,r_n \in \R$ and let
\[
g(z) = z^{2m}+1 = \prod_{j=1}^{2m}(z+\omega^{2j-1})
\]
where $\omega=\exp(2\pi i /4m)$ and $m \in \N$. Calculating as
in \eqref{eqn:Phiproductsum} gives
\begin{align*}
\Phi(z,t) & = \prod_{i=1}^n \left((z+r_i)^{2m}+t^{2m}\right)  \\
& = f(z)^{2m} \prod_{i=1}^n\Big( 1 + \frac{t^{2m}}{(z+r_i)^{2m}}\Big) \\
& = f(z)^{2m} \sum_{k=0}^n e_k\Big(\tfrac{1}{(z+r_1)^{2m}},\ldots,\tfrac{1}{(z+r_n)^{2m}}\Big) (t^{2m})^k,
\end{align*}
where $e_k$ is the $k$th elementary symmetric function of $n$
variables evaluated at $(z+r_1)^{-2m},\ldots,(z+r_n)^{-2m}$.
Thus, if $x \in \R$,
\[
A_{2mk}(x) = f(x)^{2m} e_k\Big(\tfrac{1}{(x+r_1)^{2m}},\ldots,\tfrac{1}{(x+r_n)^{2m}}\Big)
\]
is expressed as a sum of squares of real numbers and is
therefore nonnegative. Dilcher and Stolarsky studied the zeros
of $A_{2mk}(x)$. (See Prop.~2.3 and \S3
of~\cite{DilcherStolarsky1992}).
\end{example}

\begin{example}
This example illustrates how certain modifications to
Theorem~\ref{theorem:main} are possible. Let $f(z)$ be a
polynomial with \textit{negative} roots. Then
$f(z)=\prod_{i=1}^n(z+r_i)$ where each $r_i > 0$. Let
\[
g(z)=1+z+z^2 = (z+e^{\pi i /3})(z+e^{-\pi i /3}).
\]
Although $g(z)$ is not even as in the hypothesis of the
theorem, its coefficients are nonnegative. Then
\begin{align*}
\Phi(z,t) & = f(z+te^{\pi i/3})f(z+te^{-\pi i/3}) \\
& = \underbrace{f(z)^2}_{A_0(z)} + \underbrace{f(z)f'(z)}_{A_1(z)} t + \underbrace{\tfrac{1}{2!}\big(2f'(z)^2-f(z)f''(z)\big)}_{A_2(z)} t^2 \\
& \qquad + \underbrace{\tfrac{1}{3!}\bigl(3 f'(z)f''(z)-2f(z)f'''(z)\big)}_{A_3(z)}t^3  \\
& \qquad + \underbrace{\tfrac{1}{4!}\bigl(6 f''(z)^2-4f'(z)f^{(3)}(z)-f(z)f^{(4)}(z)\big)}_{A_4(z)} t^4 + \cdots
\end{align*}
On the other hand, calculating as in \eqref{eqn:Phiproductsum} gives
\begin{align*}
\Phi(z,t)
& = \prod_{i=1}^n \Big( (z+r_i)^2+(z+r_i)t+t^2\Big) \\
& = f(z)^2 \prod_{i=1}^n \left(1+\frac{t}{z+r_i}+\frac{t^2}{(z+r_i)^2}\right) \\
& = f(z)^2  \sum_{k=0}^{2n}\Big( \sum_{\substack{\lambda \vdash k \\ \lambda_j \leq 2}} m_{\lambda}(\tfrac{1}{z+r_1},\ldots,\tfrac{1}{z+r_n})\Big) t^k,
\end{align*}
where the inner sum is over all unordered partitions $\lambda$
of $k$ whose parts satisfy $\lambda_j \leq 2$ and where
$m_{\lambda}$ is the monomial symmetric function in $n$
variables for the partition $\lambda$ evaluated at
$(z+r_1)^{-1},\ldots,(z+r_n)^{-1}$. From the last expression,
we see that each
\[
A_k(x) \geq 0
\]
for all $x \geq 0$ and all $k \geq 0$.
\end{example}

\bibliographystyle{ws-procs9x6}

\begin{thebibliography}{1}

\bibitem{CardonRich2008}
David~A. Cardon and Adam Rich, \emph{Tur\'an inequalities and subtraction-free
  expressions}, JIPAM. J. Inequal. Pure Appl. Math. \textbf{9} (2008), no.~4,
  Article 91, 11 pp. (electronic).

\bibitem{CravenCsordas2002}
Thomas Craven and George Csordas, \emph{Iterated {L}aguerre and {T}ur\'an
  inequalities}, JIPAM. J. Inequal. Pure Appl. Math. \textbf{3} (2002), no.~3,
  Article 39, 14 pp. (electronic).

\bibitem{CsordasEscassut2005}
George Csordas and Alain Escassut, \emph{The {L}aguerre inequality and the
  distribution of zeros of entire functions}, Ann. Math. Blaise Pascal
  \textbf{12} (2005), no.~2, 331--345.

\bibitem{CsordasVarga1990}
George Csordas and Richard~S. Varga, \emph{Necessary and sufficient conditions
  and the {R}iemann hypothesis}, Adv. in Appl. Math. \textbf{11} (1990), no.~3,
  328--357.

\bibitem{DilcherStolarsky1992}
Karl Dilcher and Kenneth~B. Stolarsky, \emph{On a class of nonlinear
  differential operators acting on polynomials}, J. Math. Anal. Appl.
  \textbf{170} (1992), no.~2, 382--400.

\bibitem{Levin1980}
Boris~Ja. Levin, \emph{Distribution of zeros of entire functions}, revised ed.,
  Translations of Mathematical Monographs, vol.~5, American Mathematical
  Society, Providence, R.I., 1980.

\bibitem{Patrick1973}
Merrell~L. Patrick, \emph{Extensions of inequalities of the {L}aguerre and
  {T}ur\'an type}, Pacific J. Math. \textbf{44} (1973), 675--682.

\bibitem{RahmanSchmeisser2002}
Qazi~I. Rahman and Gerhard Schmeisser, \emph{Analytic theory of polynomials},
  London Mathematical Society Monographs. New Series, vol.~26, The Clarendon
  Press Oxford University Press, Oxford, 2002.

\end{thebibliography}

\end{document}